\documentclass[a4paper]{amsart}
\usepackage[active]{srcltx}
\usepackage[all]{xy}
\usepackage{subfig}
\usepackage{hyperref}
\usepackage{mathrsfs}

\usepackage{esint}
\usepackage{amsmath}
\usepackage{amssymb,latexsym}
\usepackage{mathrsfs}
\usepackage{graphics}
\usepackage{latexsym}
\usepackage{psfrag}
\usepackage{import}
\usepackage{verbatim}
\usepackage{graphicx}
\usepackage[usenames]{color}
\usepackage{pifont,marvosym}

\theoremstyle{plain}
\newtheorem{lemma}{Lemma}[section]
\newtheorem{theorem}[lemma]{Theorem}
\newtheorem{proposition}[lemma]{Proposition}
\newtheorem{corollary}[lemma]{Corollary}

\theoremstyle{definition}

\newtheorem{definition}[lemma]{Definition}

\numberwithin{equation}{section}

\newcommand{\dom}{\textrm{Dom\,}}

\newcommand{\R}{\mathbb{R}}
\newcommand{\N}{\mathbb{N}}

\newcommand{\supp}{\text{\rm supp}}

\newcommand{\Lip}{\mathrm{Lip}}

\newcommand{\diam}{\rm{diam\,}}

\newcommand{\ve}{\varepsilon}

\newcommand{\erre}{\mathbb{R}}
\newcommand{\cI}{\mathcal{I}}

\renewcommand{\r}{\varrho}

\renewcommand{\L}{\mathcal{L}}
\newcommand{\RCD}{\mathsf{RCD}}

\newcommand{\CD}{\mathsf{CD}}
\newcommand{\Geo}{{\rm Geo}}

\newcommand{\mm}{\mathfrak m}
\newcommand{\qq}{\mathfrak q}

\newcommand{\sfd}{\mathsf d}
\newcommand{\cP}{\mathcal P}
\newcommand{\PP}{\mathsf{P}}
\newcommand{\Opt}{\mathrm{OptGeo}}

\begin{document}

\title[Isoperimetric inequalities for finite perimeter sets in m.m.s.] {Isoperimetric inequalities for finite perimeter sets \\ under lower Ricci curvature bounds}
\author{Fabio Cavalletti}  \address{Dipartimento di Matematica,  Universit\'a degli Studi di Pavia, Italy}
\email{fabio.cavalletti@unipv.it}
\author{Andrea Mondino}  \address{Mathematics Institute, The University of Warwick, United Kingdom}
\email{A.Mondino@warwick.ac.uk}
%
%\address{Centro de Giorgi - SNS}
%\email{fabio.cavalletti@sns.it}

%\keywords{optimal transport; existence of maps; uniqueness of maps; measure contraction property}
%

\bibliographystyle{plain}

\begin{abstract}
We prove that the results regarding the Isoperimetric inequality and Cheeger constant formulated in terms of the \emph{Minkowski content}, 
obtained by the authors in previous papers \cite{CM1,CM2} in the framework of essentially non-branching 
metric measure spaces verifying the local curvature dimension condition, 
also hold in the stronger formulation in terms of the \emph{perimeter}. 
\end{abstract}

\maketitle
%\tableofcontents

%%%%%%%%%%%%%%%%%%%%%%%%%%%%%%%%%%%%%%%%%%%%%%%%%%
%%%%%%%%%%%%%%%%%%%%%%%%%%%%%%%%%%%%%%%%%%%%%%%%%%
%%%%%%%%%%%%%%%%%%%%%%%%%%%%%%%%%%%%%%%%%%%%%%%%%%
%%%%%%%%%%%%%%%%%%%%%%%%%%%%%%%%%%%%%%%%%%%%%%%%%%

\section{Introduction}

In the recent paper \cite{CM1} the authors proved the following sharp isoperimetric inequality: 
\begin{equation}\label{E:isouter}
\mm^{+}(E) \geq \mathcal{I}_{K,N,D}(\mm(E)),
\end{equation}
where $E \subset X$ is any Borel set, $(X,\sfd,\mm)$ is a metric measure space of diameter less than $D$
verifying the local curvature dimension condition with parameters $K$ and $N$, it is moreover essentially non-branching 
and finally $\mm$ is a probability measure. 
On the right-hand side of \eqref{E:isouter} there is $\mathcal{I}_{K,N,D}$, the sharp model isoperimetric profile function associated to $K,N,D$, 
see Section \ref{SS:IKND} for details. 
On the left hand side, $\mm^+(E)$ denotes the outer Minkowski content  of $E$: 
\begin{eqnarray}
\mm^+(E)&:=& \liminf_{\ve\downarrow 0} \frac{\mm(E^\ve)-\mm(E)}{\ve},   \label{eq:defm+}
\end{eqnarray}
where  $E^\ve:=\{x\in X\,:\,   \sfd(E, x)\leq \ve \}$. Outer Minkowski content gives a measurement of the size of $\partial E$. 
It is anyway a less accurate measurement than the one given by the perimeter.

In the Euclidean space, sets of finite perimeter are those subsets whose characteristic functions have finite total variation, in the  $BV$-sense.
If this is the case, the total variation of the distributional derivative of the characteristic function is a positive finite measure called perimeter measure; 
the perimeter of the set is the total mass of the perimeter measure.
Sets of finite perimeter can also be defined via relaxation in the following equivalent form:
given a Borel subset $E \subset \R^{n}$ and $A$ open, the perimeter of $E$ in $A$, $\mathsf{P}(E,A)$, is defined as follows
\begin{eqnarray}\label{E:perimeter-def}
\PP(E,A)&:=& \inf\left\{\liminf_{n\to \infty} \int_A |\nabla u_n| \,dx \,:\,  u_n \in \Lip(A), \, u_n\to \chi_E \text{ in } L^1_{loc}(A)\right\}  \label{eq:defP}, 
\end{eqnarray}
where  $\chi_E$ is the characteristic  function of $E$; accordingly $E \subset \R^{n}$ has finite perimeter in $\R^{n}$ if and only if $\PP(E,\R^{n}) < \infty$. 
Here $\Lip(A)$ denotes the space of real valued Lipschitz functions defined over $A$. 

Looking at \eqref{E:perimeter-def}, one finds out that all the objects used to define $\PP(E,A)$ 
have a clear generalization when we substitute the Euclidian space with  any metric space and the Lebesgue measure 
with any Borel measure. The theory of sets with finite perimeter in metric spaces has been developed in great generality in \cite{Am1, Am2}, 
to which we refer for a deeper insight. \\
The perimeter function $\PP(E,\cdot)$, both in the Euclidean and in the metric framework,
enjoys many nice properties: it is the restriction to open sets of a Borel measure having support in the essential boundary of $E$, $\partial^{*} E$; 
it is absolutely continuous with respect to the Hausdorff measure of codimension 1 restricted on $\partial^{*}E$ 
with density bounded from below and from above; it is used in the coarea formula, etc.. 
It is therefore more natural to look for inequalities involving $\PP(E)$ rather than $\mm^{+}(E)$. 
Moreover already from their definition one observes that $\PP(E) \leq \mm^{+} (E)$: indeed
in the definition of $\mm^{+}$ only the uniform approximation of $E$ is considered while in $\PP(E)$  any Lipschitz approximation  is allowed. 

It appears then as a natural question whether or not \eqref{E:isouter} holds true replacing $\mm^{+}(E)$ with $\PP(E)$.
The scope of this note is to answer affirmatively to  this question: we generalise the results proved in \cite{CM1} to the perimeter case.

\begin{theorem}\label{T:iso}
Let $(X,\sfd,\mm)$ be a metric measure space with $\mm(X)=1$, 
verifying  the essentially non-branching property and $\CD_{loc}(K,N)$ for some $K\in \R,N \in [1,\infty)$.
Let $D$ be the diameter of $X$, possibly assuming the value $\infty$.
\medskip

Then for every Borel subset $E\subset X$, calling $\mm(E)=v\in [0,1]$, it holds
$$
\PP(E) \ \geq \ \cI_{K,N,D}(v).
$$

\end{theorem}

The proof of Theorem \ref{T:iso} follows the same scheme of the proof of \eqref{E:isouter} contained in \cite{CM1}. 
There, the analysis on the one dimensional version of \eqref{E:isouter} permitted to obtain the general \eqref{E:isouter} 
via a one-dimensional localization argument based on an $L^{1}$-Optimal Transportation problem (see also \cite{klartag}). 
Here again we first prove Theorem \ref{T:iso} in the easier one-dimensional framework and then we obtain the general case via localization. 
The one dimensional analysis is contained in Section \ref{S:onedimensional} while Section \ref{S:main} 
contains the proofs of Theorem \ref{T:iso}, of the Cheeger isoperimetric inequality (Theorem \ref{T:cheeger})
and of the corresponding almost rigidity result (Corollary \ref{C:almostcheeger}).

We also include below the statements of rigidity, the almost maximal diameter and the almost rigidity results one can obtain by replacing 
the outer Minkowski content with the perimeter. 
\medskip

Repeating verbatim the proof of  \cite[Theorem 1.4]{CM1} we obtain the rigidity for the perimeter. 
We set
$$
\cI_{(X,\sfd,\mm)}(v):=\inf\{\PP(E)\,: \, E \subset X, \, \mm(E)=v\} .
$$

\begin{theorem}\label{thm:Rigidity}
Let $(X,\sfd,\mm)$ be an  $\RCD^*(N-1,N)$ space for some  $N \in [2,\infty)$, with  $\mm(X)=1$. 
Assume that there exists $\bar{v} \in (0,1)$ such that $\cI_{(X,\sfd,\mm)}(\bar{v})=\cI_{N-1,N,\infty}(\bar{v})$. 
\medskip

Then $(X,\sfd,\mm)$ is a spherical suspension:  there exists
an $\RCD^*(N-2,N-1)$ space $(Y,\sfd_{Y}, \mm_{Y})$ with $\mm_{Y}(Y)=1$ such that  $X$ is isomorphic as metric measure space to $[0,\pi] \times^{N-1}_{\sin} Y$.

Moreover, in this case, the following hold:
\begin{itemize}
\item[$i)$]  For every $v\in [0,1]$ it holds  $\cI_{(X,\sfd,\mm)}(v)=\cI_{N-1,N,\infty}(v)$. 
\item[$ii)$] For every $v\in [0,1]$ there exists a Borel subset $A \subset X$ with $\mm(A)=v$ such that 
$$
\PP(A)=\cI_{(X,\sfd,\mm)}(v)=\cI_{N-1,N,\infty}(v).
$$
\item[$iii)$]  If $\mm(A)\in (0,1)$ then  $\PP(A)=\cI_{(X,\sfd,\mm)}(v)=\cI_{N-1,N,\infty}(v)$ if and only if
$$
\min\Big(\mm(A\setminus \{(t,y)\in [0,\pi] \times^{N-1}_{\sin} Y \, :\, t \in [0,r_v] \}), \mm(A\setminus\{(t,y)\in [0,\pi] \times^{N-1}_{\sin} Y \, :\, t \in [\pi-r_v, \pi] \}) \Big)=0,
$$
where  $r_v\in (0,\pi)$ is chosen so 
that $\int_{[0,r_v]} c_N (\sin(t))^{N-1} dt =v$, $c_N$ being given by $c_N^{-1}:= \int_{[0,\pi]}  (\sin(t))^{N-1} dt$.
\end{itemize}
\end{theorem}

Repeating verbatim the proof of \cite[Theorem 1.5]{CM1} we get  
\begin{theorem}[Almost equality in L\'evy-Gromov implies almost maximal diameter]\label{thm:AlmRig}
For every  $N>1$, $v \in (0,1)$, $\ve>0$ there exists $\bar{\delta}=\bar{\delta}(N,v,\ve)>0$ such that the following holds. 
For every $\delta\in [0, \bar{\delta}]$,  if $(X,\sfd,\mm)$ is an $\RCD^*(N-1-\delta,N+\delta)$ space satisfying 
$$
\cI_{(X,\sfd,\mm)}(v)\leq \cI_{N-1,N,\infty}(v)+\delta, 
$$
Then $\diam((X,\sfd)) \geq \pi-\ve$.
\end{theorem}

The following corollary is a consequence of the Maximal Diameter Theorem \cite{Ket}, and of  the compactness/stability of the class of $\RCD^*(K,N)$ spaces, for some fixed $K>0$ and $N>1$,  with respect to the measured Gromov-Hausdorff convergence (for more details see \cite[Section 6.4]{CM1}).  We denote by $\sfd_{mGH}$ the measured-Gromov Hausdorff distance between compact probability metric measure spaces.

\begin{corollary}[Almost equality in L\'evy-Gromov implies mGH-closeness to a spherical suspension] \label{cor:AlmRig}
For every $N\in [2, \infty) $, $v \in (0,1)$, $\ve>0$ there exists $\bar{\delta}=\bar{\delta}(N,v,\ve)>0$ such that the following hold. 
For every  $\delta \in [0, \bar{\delta}]$, if  $(X,\sfd,\mm)$ is an $\RCD^*(N-1-\delta,N+\delta)$ space satisfying 
$$
\cI_{(X,\sfd,\mm)}(v)\leq \cI_{N-1,N,\infty}(v)+\delta, 
$$
then  there exists an $\RCD^*(N-2,N-1)$ space $(Y, \sfd_Y, \mm_Y)$ with $\mm_Y(Y)=1$ such that 
$$
\sfd_{mGH}(X, [0,\pi] \times_{\sin}^{N-1} Y) \leq \ve. 
$$
\end{corollary}

Let  us finally mention the  closely related independent preprint of
Ambrosio, Gigli and Di Marino \cite{AGD}, where it is proved that on general metric measure spaces the perimeter is equal to the relaxation of the Minkowski content w.r.t. convergence in measure.

\section*{Acknowledgements}
Part  of the work has been  developed while A. M.  was in residence at the Mathematical Science Research Institute in Berkeley, California, during the spring  semester 2016 and was  supported by the National Science Foundation under the Grant No. DMS-1440140 and part when A. M. was supported by the EPSRC First Grant EP/R004730/1.

\section{Preliminaries}

The space of all Borel probability measures over $X$ will be denoted by $\mathcal{P}(X)$.
A metric space is a geodesic space if and only if for each $x,y \in X$ 
there exists $\gamma \in \Geo(X)$ so that $\gamma_{0} =x, \gamma_{1} = y$, with
$$
\Geo(X) : = \{ \gamma \in C([0,1], X):  \sfd(\gamma_{s},\gamma_{t}) = |s-t| \sfd(\gamma_{0},\gamma_{1}), \text{ for every } s,t \in [0,1] \}.
$$
Recall that for complete geodesic spaces local compactness is equivalent to properness (a metric space is proper if every closed ball is compact).
The most general case of spaces we will consider are essentially non-branching metric measure spaces $(X,\sfd,\mm)$ verifying $\CD_{loc}(K,N)$;
it is therefore not restrictive to assume that $\supp(\mm) = X$ and $(X,\sfd)$ to be proper and geodesic. 
Hence we will assume that the ambient metric space $(X, \sfd)$ is geodesic, complete, separable and proper and $\mm(X) = 1$.
We denote by $\Lip(X)$ the space of real-valued Lipschitz functions over $X$. Given $u \in \Lip(X)$ its slope $|\nabla u|(x)$ at $x\in X$ is defined by
\begin{equation}
|\nabla u|(x):=\limsup_{y\to x} \frac{|u(x)-u(y)|}{\sfd(x,y)}.
\end{equation}  
Following \cite{Am1,Am2,Mir} and the more recent \cite{ADM},
given a Borel subset $E \subset X$ and $A$ open, the perimeter $\mathsf{P}(E,A)$ is defined as follows
\begin{eqnarray}
\PP(E,A)&:=& \inf\left\{\liminf_{n\to \infty} \int_A |\nabla u_n| \,\mm \,:\,  u_n \in \Lip(A), \, u_n\to \chi_E \text{ in } L^1(A,\mm)\right\}.  \label{eq:defP} 
\end{eqnarray}
We say that $E \subset X$ has finite perimeter in $X$ if $\PP(E,X) < \infty$. We recall also few properties of the perimeter functions:
\begin{itemize}
\item[(a)] (locality) $\PP(E,A) = \PP(F,A)$, whenever $\mm(E\Delta F \cap A) = 0$;
\item[(b)] (l.s.c.) the map $E \mapsto \PP(E,A)$ is lower-semicontinuous with respect to the $L^{1}_{loc}(A)$ convergence;
\item[(c)] (complementation) $\PP(E,A) = \PP(E^{c},A)$.
\end{itemize}
Most importantly, if $E$ is a set of finite perimeter, then the set function $A \to \PP(E,A)$ is the restriction to open sets of a finite 
Borel measure $\PP(E,\cdot)$ in $X$ (see Lemma 5.2 of \cite{ADM}), defined by 
$$
\PP(E,B) : = \inf \{ \PP(E,A) \colon A \supset B, \ A\ \textrm{open} \}.
$$
Sometimes, for ease of notation, we will write $\PP(E)$ instead of $\PP(E,X)$. 
The outer Minkowski content  $\mm^+(E)$ of $E$ are defined respectively by
\begin{eqnarray}
\mm^+(E)&:=& \liminf_{\ve\downarrow 0} \frac{\mm(E^\ve)-\mm(E)}{\ve},   \label{eq:defm+}
\end{eqnarray}
where  $\chi_E$ is the characteristic  function of $E$, and  $E^\ve:=\{x\in X\,:\,   \sfd(E, x)\leq \ve \}$. 
With a slight abuse of notation we denoted $\sfd(E,x):=\inf_{y \in E} \sfd(y,x)$.

It is an immediate consequence of the definition that for open sets $\PP(E)\leq \mm^+(E)$; let give a short proof for the reader's convenience.
For the result in the smooth framework of Riemannian manifolds, see for instance \cite{BurZal}.

\begin{proposition}[$\PP(E)\leq \mm^+(E)$] \label{prop:P<m+}
Let $(X,\sfd,\mm)$ be a m.m.s. as above and $E\subset X$ be an open set. Then  $\PP(E)\leq \mm^+(E)$.
\end{proposition}

\begin{proof}
Let $\ve_n\downarrow 0$ be such that 
\begin{equation}   \label{eq:defm+}
\mm^+(E) = \lim_{\ve_n\downarrow 0} \frac{\mm(E^{\ve_n})-\mm(E)}{\ve_n},
\end{equation}
and define $u_n(x):=\max\{0, 1-\ve_n^{-1} \sfd(x,E)\}$. 
Notice that if $x \in E$, since $E$ is open then  $u_n\equiv 1$ on a small ball around $x$ and therefore $|\nabla u_n|\equiv 0$ on $E$.  
Moreover $u_n \equiv 0 $ on $X\setminus E^{\ve_n}$ so, since $X\setminus E^{\ve_n}$ is open, 
we get  $|\nabla u_n|\equiv 0$ on $X\setminus E^{\ve_n}$. 
Finally it is clear from the definition that $|\nabla u_n|\leq \ve_n^{-1}$ on $E^{\ve_n}\setminus E$. Combining these informations we obtain
$$
\int_X |\nabla u_n| \, \mm \leq \frac{\mm(E^{\ve_n}\setminus E)}{\ve_n}.
$$ 
Now we can pass to the limit as $n \to \infty$ and use  \eqref{eq:defm+} to obtain the thesis.
\end{proof}

%Since for open sets $\PP(\cdot)\leq \mm^+(\cdot)$,  a priori an isoperimetric inequality expressed in terms of the perimeter is stronger  than 
%an isoperimetric inequality expressed in terms of the exterior Minkowski content. 
%As in our previous paper \cite{CM1} we established sharp and rigid isoperimetric inequalities in metric 
%measure spaces with Ricci curvature bounded from below in terms of  the exterior Minkowski content, 
%it is a natural question whether the same inequalities (with associated rigidity and almost rigidity statements) 
%remain valid if we replace the perimeter with the Minkowski content. The answer is affirmative and the goal of this short note is to prove this fact.

\subsection{Geometry of metric measure spaces}\label{Ss:geom}
Here we briefly recall the synthetic notions of lower Ricci curvature bounds, for more detail we refer to  \cite{BS10,lottvillani:metric,sturm:I, sturm:II, Vil}.

%A \emph{metric measure space} is a triple $(X,\sfd,\mm)$ where 
%$(X,\sfd)$ is a complete separable metric space and $\mm$ is a locally finite measure (i.e. $m(B_{r}(x))< \infty$ for all $x\in X$ 
%and all sufficiently small $r>$0) on $X$ equipped with its Borel $\sigma$-algebra. 
%
%The space of Borel probability measures on $X$ with finite second moment is denoted with $\mathcal{P}_{2}(X,\sfd)$ 
%and $W_{2}$ denotes the $L^{2}$-Wasserstein distance.  
%The following are well-known results in optimal transportation theory and are valid for general metric measure spaces.
%\begin{lemma}\label{L:geod}
%Let $(X,d,m)$ be a metric measure space.
%For each geodesic $\mu: [0,1] \to \mathcal{P}_{2}(X,d)$ there exists a probability measure $\gammA$ on $\G(X)$ such that 
%\begin{itemize}
%\item $e_{t\,\sharp} \gammA = \mu_{t}$ for all $t \in [0,1]$;
%\item for each pair $(s,t)$ the transference plan $(e_{s},e_{t})_{\sharp} \gammA$ is an optimal coupling for $W_{2}$.
%\end{itemize}
%\end{lemma}
%Consider the R\'enyi entropy functional 
%\[
%\S_{N}(\, \cdot\, | m ) : \mathcal{P}_{2}(X,d) \to \erre
%\]
%with respect to $m$, defined by
%\begin{equation}\label{E:entropy}
%\mathcal{S}_{N}(\mu | m) : = - \int_{X} \r^{-1/N}(x) \mu(dx)
%\end{equation}
%for $\mu \in \mathcal{P}_{2}(X)$, where $\r$ is the density of the absolutely continuous part $\mu^{c}$ in the Lebesgue decomposition 
%$\mu = \mu^{c} + \mu^{s} = \r m + \mu^{s}$.
%

In order to formulate curvature properties for $(X,\sfd,\mm)$ we introduce the following distortion coefficients: given two numbers $K,N\in \erre$ with $N\geq0$, we set for $(t,\theta) \in[0,1] \times \erre_{+}$, 
\begin{equation}\label{E:sigma}
\sigma_{K,N}^{(t)}(\theta):= 
\begin{cases}
\infty, & \textrm{if}\ K\theta^{2} \geq N\pi^{2}, \crcr
\displaystyle  \frac{\sin(t\theta\sqrt{K/N})}{\sin(\theta\sqrt{K/N})} & \textrm{if}\ 0< K\theta^{2} <  N\pi^{2}, \crcr
t & \textrm{if}\ K \theta^{2}<0 \ \textrm{and}\ N=0, \ \textrm{or  if}\ K \theta^{2}=0,  \crcr
\displaystyle   \frac{\sinh(t\theta\sqrt{-K/N})}{\sinh(\theta\sqrt{-K/N})} & \textrm{if}\ K\theta^{2} \leq 0 \ \textrm{and}\ N>0.
\end{cases}
\end{equation}

We also set, for $N\geq 1, K \in \R$ and $(t,\theta) \in[0,1] \times \erre_{+}$
\begin{equation} \label{E:tau}
\tau_{K,N}^{(t)}(\theta): = t^{1/N} \sigma_{K,N-1}^{(t)}(\theta)^{(N-1)/N}.
\end{equation}

%\begin{equation}\label{E:tau}
%\tau_{K,N}^{(t)}(\theta):= 
%
%\begin{cases}
%
%\infty, & \textrm{if}\ K\theta^{2} \geq (N-1)\pi^{2}, \crcr
%
%\displaystyle  t^{\frac{1}{N}}\Bigg(\frac{\sin(t\theta\sqrt{K/(N-1)})}{\sin(\theta\sqrt{K/(N-1)})}\Bigg)^{\frac{N-1}{N}} & \textrm{if}\ 0< K\theta^{2} < (N-1)\pi^{2}, \crcr
%
%t & \textrm{if}\ K \theta^{2}<0 \ \textrm{and}\ N=1\ \textrm{or}\\& \textrm{if}\ K \theta^{2}=0,  \crcr
%
%\displaystyle  t^{\frac{1}{N}}\Bigg(\frac{\sinh(t\theta\sqrt{-K/(N-1)})}{\sinh(\theta\sqrt{-K/(N-1)})}\Bigg)^{\frac{N-1}{N}} & \textrm{if}\ K\theta^{2} \leq 0 \ \textrm{and}\ N>1.
%
%\end{cases}
%
%\end{equation}
%That is, $\tau_{K,N}^{(t)}(\theta): = t^{1/N} \sigma_{K,N-1}^{(t)}(\theta)^{(N-1)/N}$ where
%
%$$
%\sigma_{K,N}^{(t)}(\theta) = \frac{\sin(t\theta\sqrt{K/N})}{\sin(\theta\sqrt{K/N})}, 
%$$
%
%if $0 < K\theta^{2}<N\pi^{2}$ and with appropriate interpretation otherwise. 

As we will consider only the case of essentially non-branching spaces, we recall the following definition. 
\begin{definition}\label{D:essnonbranch}
A metric measure space $(X,\sfd, \mm)$ is \emph{essentially non-branching} if and only if for any $\mu_{0},\mu_{1} \in \mathcal{P}_{2}(X)$,
with $\mu_{0}, \mu_{1}$ absolutely continuous with respect to $\mm$, any element of $\Opt(\mu_{0},\mu_{1})$ is concentrated on a set of non-branching geodesics.
\end{definition}

A set $F \subset \Geo(X)$ is a set of non-branching geodesics if and only if for any $\gamma^{1},\gamma^{2} \in F$, it holds:
$$
%\gamma_{0}^{1} = \gamma_{0}^{2}, \; 
\exists \;  \bar t\in (0,1) \text{ such that } \ \forall t \in [0, \bar t\,] \quad  \gamma_{ t}^{1} = \gamma_{t}^{2}   
\quad 
\Longrightarrow 
\quad 
\gamma^{1}_{s} = \gamma^{2}_{s}, \quad \forall s \in [0,1].
$$

The classic definition of $\CD(K,N)$ given in \cite{lottvillani:metric,sturm:I, sturm:II} can be rewritten as follows (see \cite{CM:CCM}).

\begin{definition}[$\CD$ condition]\label{D:CD}
An essentially non-branching m.m.s. $(X,\sfd,\mm)$ verifies $\mathsf{CD}(K,N)$  if and only if for each pair 
$\mu_{0}, \mu_{1} \in \mathcal{P}_{2}(X,\sfd,\mm)$ there exists $\nu \in \Opt(\mu_{0},\mu_{1})$ such that for all $t \in [0,1]$,
\begin{equation}\label{E:CD}
\r_{t}^{-1/N} (\gamma_{t}) \geq  \tau_{K,N}^{(1-t)}(\sfd( \gamma_{0}, \gamma_{1}))\r_{0}^{-1/N}(\gamma_{0}) 
 + \tau_{K,N}^{(t)}(\sfd(\gamma_{0},\gamma_{1}))\r_{1}^{-1/N}(\gamma_{1}), \qquad 
\end{equation}
for $\nu$-a.e. $\gamma \in \Geo(X)$, where $({\rm e}_{t})_\sharp \, \nu = \r_{t} \mm$.
\end{definition}

It is worth recalling that if $(M,g)$ is a Riemannian manifold of dimension $n$ and 
$h \in C^{2}(M)$ with $h > 0$, then the m.m.s. $(M,g,h \, vol)$ verifies $\CD(K,N)$ with $N\geq n$ if and only if  (see Theorem 1.7 of \cite{sturm:II})
$$
Ric_{g,h,N} \geq  K g, \qquad Ric_{g,h,N} : =  Ric_{g} - (N-n) \frac{\nabla_{g}^{2} h^{\frac{1}{N-n}}}{h^{\frac{1}{N-n}}}.  
$$
In particular, if $I \subset \R$ is any interval, $h \in C^{2}(I)$ 
and $\mathcal{L}^{1}$ is the one-dimensional Lebesgue measure, the m.m.s. $(I ,|\cdot|, h \mathcal{L}^{1})$ verifies $\CD(K,N)$ if and only if  
\begin{equation}\label{E:CD-N-1}
\left(h^{\frac{1}{N-1}}\right)'' + \frac{K}{N-1}h^{\frac{1}{N-1}} \leq 0.
\end{equation}
%
%This short remark motivates the definition. 

We also mention the more recent Riemannian curvature dimension condition $\RCD^{*}$. This  consists in an   enforcement of the so called reduced curvature dimension condition, denoted by $\CD^{*}(K,N)$ and introduced in \cite{BS10}: the additional condition is that the Sobolev space $W^{1,2}(X,\mm)$ is an Hilbert space, see \cite{AGS11a, AGS11b,AGMR12}.
Remarkable features of the $\RCD^{*}(K,N)$ condition are their rectifiability and   the equivalence with the dimensional Bochner inequality  \cite{EKS, AMS}.

The reduced $\CD^{*}(K,N)$ condition asks for the same inequality \eqref{E:CD} of $\CD(K,N)$ but  the
coefficients $\tau_{K,N}^{(t)}(\sfd(\gamma_{0},\gamma_{1}))$ and $\tau_{K,N}^{(1-t)}(\sfd(\gamma_{0},\gamma_{1}))$ 
are replaced by $\sigma_{K,N}^{(t)}(\sfd(\gamma_{0},\gamma_{1}))$ and $\sigma_{K,N}^{(1-t)}(\sfd(\gamma_{0},\gamma_{1}))$, respectively.
For both definitions there is a local version; here we only state the one for $\mathsf{CD}(K,N)$, being clear what would be the one for $\mathsf{CD}^{*}(K,N)$.

\begin{definition}[$\CD_{loc}$ condition]\label{D:loc}
An essentially non-branching m.m.s. $(X,\sfd,\mm)$ satisfies $\CD_{loc}(K,N)$ if for any point $x \in X$ there exists a neighbourhood $X(x)$ of $x$ such that for each pair 
$\mu_{0}, \mu_{1} \in \mathcal{P}_{2}(X,\sfd,\mm)$ supported in $X(x)$
there exists $\nu \in \Opt(\mu_{0},\mu_{1})$ such that \eqref{E:CD} holds true for all $t \in [0,1]$.
The support of $({\rm e}_{t})_\sharp \, \nu$ is not necessarily contained in the neighbourhood $X(x)$.
\end{definition}

\subsection{The model Isoperimetric profile function $\cI_{K,N,D}$}\label{SS:IKND}
If $K>0$ and $N\in \N$, by the Levy-Gromov isoperimetric inequality 
we know that, for $N$-dimensional smooth manifolds having Ricci $\geq K$, the isoperimetric profile function is bounded below by the one of the $N$-dimensional round sphere of the suitable radius. In other words  the \emph{model} isoperimetric profile function is the one of ${\mathbb S}^N$. For $N\geq 1, K\in \R$ arbitrary real numbers the situation is  more complicated, and just recently E. Milman \cite{Mil} discovered what is the model isoperimetric profile. In this short section we recall its definition.
\\

Given $\delta>0$, set 
\[
\begin{array}{ccc}
 s_\delta(t) := \begin{cases}
\sin(\sqrt{\delta} t)/\sqrt{\delta} & \delta > 0 \\
t & \delta = 0 \\
\sinh(\sqrt{-\delta} t)/\sqrt{-\delta} & \delta < 0
\end{cases}

& , &

 c_\delta(t) := \begin{cases}
\cos(\sqrt{\delta} t) & \delta > 0 \\
1 & \delta = 0 \\
\cosh(\sqrt{-\delta} t) & \delta < 0
\end{cases}
\end{array} ~.
\]
Given a continuous function $f :\R \to  \R$ with $f(0) \geq 0$, we denote by $f_+ : \R \to \R^+ $ the function coinciding with $f$ between its first non-positive and first positive roots, and vanishing everywhere else, i.e. $f_+ := f \chi_{[\xi_{-},\xi_{+}]}$ with $\xi_{-} = \sup\{\xi \leq 0; f(\xi) = 0\}$ and $\xi_{+} = \inf\{\xi > 0; f(\xi) = 0\}$.

Given $H,K \in \R$ and $N \in [1,\infty)$, set $\delta := K / (N-1)$  and define the following (Jacobian) function of $t \in \R$:
\begin{equation}\label{def:J}
J_{H,K,N}(t) :=
\begin{cases}
\chi_{\{t=0\}}  & N = 1 , K > 0 \\
\chi_{\{H t \geq 0}\} & N = 1 , K \leq 0 \\
\left(c_\delta(t) + \frac{H}{N-1} s_\delta(t)\right)_+^{N-1} & N \in (1,\infty) \\
\end{cases} ~.
\end{equation}
As last piece of notation, given a non-negative integrable function $f$ on a closed interval $L \subset \R$, we denote with $\mu_{f,L}$  
the probability measure supported in $L$ with density (with respect to the Lebesgue measure) proportional to $f$ there. In order to simplify a bit the notation we will write
$\cI_{(L,f)}$ in place of $\cI_{(L,\, |\cdot|, \mu_{f,L})}$.
\\The model isoperimetric profile for spaces having Ricci $\geq K$, for some $K\in \R$, dimension bounded above by $N\geq 1$ and diameter at most $D\in (0,\infty]$ is then defined by
\begin{equation}\label{eq:defIKND}
\cI_{K,N,D}(v):=\inf_{H\in \R,a\in [0,D]} \cI_{\left([-a,D-a], J_{H,K,N}\right)} (v), \quad \forall v \in [0,1]. 
\end{equation}
The formula above has the advantage of considering all the possible cases in just one equation, for the explicit discussion of the different cases we refer to  \cite[Section 4] {Mil}. Here let us just note that when $N$ is an integer, $$\cI_{\big( [0,  \sqrt{\frac{N-1}{K}} \pi ], ( \sin(\sqrt{\frac{K}{N-1} } t)^{N-1}\big)} = \cI_{({\mathbb S}^{N}, g^K_{can}, \mu^K_{can})}$$ by the isoperimetric inequality on the sphere, so  the case $K>0$  with $N$ integer corresponds to L\'evy-Gromov isoperimetric inequality.

\subsection{1-D localization}

Before stating the next result let us recall that  $\CD^{*}(K,N)$ and $\CD_{loc}(K,N)$ are equivalent if  $1 < N <\infty$ or $N =1$ and $K \geq 0$, but for $N =1$ and $K < 0$ the $\CD_{loc}(K,N)$ condition is strictly stronger than $\CD^{*}(K,N)$.

\begin{theorem}\cite[Theorem 5.1]{CM1}\label{T:localize}
Let $(X,\sfd, \mm)$ be an essentially non-branching metric measure space verifying the $\CD_{loc}(K,N)$ condition for some $K\in \R$ and $N\in [1,\infty)$. 
Let $f : X \to \R$ be $\mm$-integrable such that $\int_{X} f\, \mm = 0$ and assume the existence of $x_{0} \in X$ such that $\int_{X} | f(x) |\,  \sfd(x,x_{0})\, \mm(dx)< \infty$. 
\medskip

Then the space $X$ can be written as the disjoint union of two sets $Z$ and $\mathcal{T}$ with $\mathcal{T}$ admitting a partition 
$\{ X_{q} \}_{q \in Q}$ and a corresponding disintegration of $\mm\llcorner_{\mathcal{T}}$, $\{\mm_{q} \}_{q \in Q}$ such that: 

\begin{itemize}
\item For any $\mm$-measurable set $B \subset \mathcal{T}$ it holds 
$$
\mm(B) = \int_{Q} \mm_{q}(B) \, \qq(dq), 
$$
where $\qq$ is a probability measure over $Q$ defined on the quotient $\sigma$-algebra $\mathcal{Q}$. 
\medskip
\item For $\qq$-almost every $q \in Q$, the set $X_{q}$ is a geodesic and $\mm_{q}$ is supported on it. 
Moreover $q \mapsto \mm_{q}$ is a $\CD(K,N)$ disintegration, i.e.  for $\qq$-a.e. $q \in Q$ the following curvature inequality holds:  
\begin{equation}\label{E:curvdensmm}
h_{q}( (1-s)  t_{0}  + s t_{1} )^{1/(N-1)}  
 \geq \sigma^{(1-s)}_{K,N-1}(t_{1} - t_{0}) h_{q} (t_{0})^{1/(N-1)} + \sigma^{(s)}_{K,N-1}(t_{1} - t_{0}) h_{q} (t_{1})^{1/(N-1)},
\end{equation}
for all $s\in [0,1]$ and for all $t_{0}, t_{1} \in \dom(g(q,\cdot))$ with  $t_{0} < t_{1}$. If $N =1$, for $\qq$-a.e. $q \in Q$ the density $h_{q}$ is constant.

\item For $\qq$-almost every $q \in Q$, it holds $\int_{X_{q}} f \, \mm_{q} = 0$ and $f = 0$ $\mm$-a.e. in $Z$.
\end{itemize}
\end{theorem} 

Let us also mention that  we can define a Borel \emph{ray map} $g: \textrm{Dom}(g)\subset Q \times \R \to X$ such that for every $q\in Q$, the map $\R\supset \dom(g(q,\cdot))\ni t\mapsto g(q,t)$ is an arc-length parametrisation of the geodesic $X_q$, i.e.
$$X_q=  g(q,\cdot) (\dom(g(q,\cdot)))  \quad \text{ and } \quad  \sfd(g(q,s), g(q,t) )= |t-s|, \quad \forall q \in Q, \forall s,t \in \dom(g(q,\cdot)).$$ 
For more details see  \cite[Section 4]{biacava:streconv},  \cite[Section 3]{CM1} and references therein.

\section{Isoperimetric inequalities in terms of the perimeter}\label{S:onedimensional}

\subsection{The one dimensional case}

Given $K\in \R, N\in[1,+\infty)$ and $D\in (0,+\infty]$, consider the following family of probability measures

\begin{eqnarray}
\mathcal{F}^{s}_{K,N,D} : = \{ \mu \in \mathcal{P}(\R) : &\supp(\mu) \subset [0,D], \, \mu = h_{\mu} \mathcal{L}^{1},\,
h_{\mu}\, \textrm{verifies} \, \eqref{E:curvdensmm} \ \textrm{and is continuous if } N\in (1,\infty), \nonumber \\ 
& \quad h_{\mu}\equiv \textrm{const} \text{ if }N=1   \}.
\end{eqnarray}
In what follows we will assume $h_{\mu}$ to be defined on the whole $\R$, vanishing outside of $\supp(\mu)$.\\
Denote with $\mathcal{I}^{s}_{K,N,D}$ the corresponding  comparison \emph{synthetic} isoperimetric profile
$$
\mathcal{I}^{s}_{K,N,D}(v) : = \inf \left\{ \mu^{+}(A) \colon A\subset \R, \,\mu(A) = v, \, \mu \in \mathcal{F}^{s}_{K,N,D}  \right\},
$$
where $\mu^{+}(A)$ denotes the Minkowski content. The term synthetic refers to $\mu \in \mathcal{F}^{s}_{K,N,D}$ 
meaning that the Ricci curvature bound is satisfied in its synthetic formulation: if $\mu = h \cdot \mathcal{L}^{1}$, then $h$ verifies \eqref{E:curvdensmm}. 
It was proved in \cite[Theorem 6.3]{CM1} that  for every $v \in [0,1]$ it holds $\mathcal{I}^{s}_{K,N,D}(v)= \mathcal{I}_{K,N,D}(v)$.

It is worth also specifying the formula \eqref{eq:defP} to the one dimensional case.
So if $(\supp(\mu), |\cdot|, \mu) \in \mathcal{F}^{s}_{K,N,D}$ and $B$ is a Borel set:
\begin{equation}\label{E:perimeter1d}
\PP_{(\supp(\mu), |\cdot|, \mu)} (B) =
\inf\left\{\liminf_{n\to \infty} \int_{\supp(\mu)} | u_n'| h_{\mu}\, \L^{1} \,:\,  u_n \in \Lip(\supp(\mu)), \, u_n\to \chi_{B} \text{ in } L^1_{loc}(\mu)\right\}.
\end{equation}
Since $\mu \in \mathcal{F}^{s}_{K,N,D}$, it follows that $\supp(\mu)$ is an interval that, up to a translation, is a subset of $[0,D]$.
Note that from the locality of the perimeter, if $E \subset \supp(\mu)$ is a Borel set, then 
$$
\PP_{(\supp(\mu), |\cdot|, \mu)}(E)= \PP_{([0,D], |\cdot|, \mu)}(E).
$$
In the next lemma we show that, in the one dimensional case, the perimeter has a precise representation.

\begin{proposition}\label{prop:P1D} 
Let $\mu= h_{\mu} \mathcal{L}^{1} \in \mathcal{F}^{s}_{K,N,D}$ 
and let $E \subset \supp(\mu)$ be a Borel subset with $\PP_{([0,D], |\cdot|, \mu)} (E)<\infty$. 
Then there exist countably many disjoint closed intervals $\{[a_i, b_i]\}_{i \in \N}$ such that  $\mu(E \triangle  \bigcup_{i \in \N} [a_i, b_i] )=0$, and 
$$
\PP_{([0,D], |\cdot|, \mu)} (E)=\sum_{i \in \N} \big( h_\mu(a_i) + h_{\mu} (b_i) \big).
%=\mu^+(E).  
$$
\end{proposition}

Note that the disjoint closed intervals given by Proposition \ref{prop:P1D} are not necessarily subsets of $\supp (\mu)$. In particular it may happen 
that $h_{\mu}(a_{i})=0$ or $h_{\mu}(b_{j})=0$ for some $i,j \in \N$. Note moreover that from the disjointness of the family of closed intervals 
verifying $\mu(E \triangle  \bigcup_{i \in \N} [a_i, b_i] )=0$ it follows also their uniqueness inside $\supp(\mu)$.

\begin{proof}

{\bf Step 1.} $E$ is countable union of intervals.\\
If $D=0$ everything trivializes so we can assume $D>0$. 
Possibly choosing a smaller $D\geq0$ and operating a shift of the interval, we can also assume 
that $0=\inf \{t\in[0,D]: h_{\mu}(t)>0)\}$ and $D=\sup \{t\in[0,D]: h_{\mu}(t)>0)\}$.

Since $1=\mu([0,D])=\int_{[0,D]} h_{\mu}(t) \mathcal{L}^{1}(dt)$, there exists $t_0 \in [0,D]$ such that $h_{\mu}(t_0)\geq 1/D$. 
From the concavity condition \eqref{E:curvdensmm} it is not difficult to check that $h_{\mu}(t)>0$  for all $t \in (0,D)$ and for every $\ve> 0$ and any $t \in [\ve,D-\ve]$ 
it holds $h_{\mu}(t)\geq C(\ve) >0$. Then from \eqref{E:perimeter1d}: 
$$
\PP_{([0,D], |\cdot|, \mu)} (E) \geq \PP_{([\ve,D-\ve], |\cdot|, \mu)} (E) \geq C(\ve)\cdot \PP_{([\ve,D-\ve], |\cdot|, \L^{1})} (E).
$$
It follows from standard results on one-dimensional sets of finite perimeter w.r.t. the Lebesgue measure (see for instance \cite{AFP}) 
that, up to an ${\mathcal L}^1$-negligible subset,  $\overline{E} \cap [\ve, D-\ve]$ 
is the finite union of disjoint closed intervals contained in $[\ve,D-\ve]$, where $\overline{E}$ denotes the closure of $E$, and 
$\mathring{E}  \cap [\ve, D-\ve]$ is the finite union of the interior of the intervals contained in $[\ve,D-\ve]$ considered before, 
where $\mathring{E}$ denotes the interior of $E$.
Repeating the same argument with a sequence of $\ve_{n} \to 0$, it follows that $\overline{E}$, up to a set of $\mu$-measure zero,
is the countable union of disjoint closed intervals, say $\{[a_{i},b_{i}]\}_{i\in \N}$.
 
%and $\mu(\bar E) = \mu(\mathring{E})$, where $\mathring{E}$ denotes the interior part of $E$.
%in particular since $\mu$ is absolutely continuous with respect to $\L^{1}$
%%
%$$
%\PP_{([0,D], |\cdot|, \mu)} (\bar E) = \PP_{([0,D], |\cdot|, \mu)} (E) = \PP_{([0,D], |\cdot|, \mu)} (\mathring{E}).
%$$
%
%and 
%%
%$$
%\mu^{+}(\bar E) \geq \mu^{+}(E) \geq \mu^{+}(\mathring{E})
%$$
%%

\medskip

{\bf Step 2.} Representation formula: we now prove the identity of the claim assuming $\cup_{i\in \N} [a_{i},b_{i}] \subset (0,D)$, 
with $a_{i}, b_{i} \to 0$ as $i \to \infty$, as the general case follows similarly. \\
Let us first prove the inequality
\begin{equation}\label{eq:P>sum}
\PP_{([0,D], |\cdot|, \mu)} (E) \geq \sum_{i\in \N} h_{\mu}(a_{i}) + h_{\mu}(b_{i}). 
\end{equation}
%We first claim that $\PP_{([0,D], |\cdot|, \mu)} (E) \geq \PP_{([\ve,D-\ve], |\cdot|, \mu)} (E)$  implies
%$$
%\PP_{([0,D], |\cdot|, \mu)} (E) \geq \sum_{i\in \N} h_{\mu}(a_{i}) + h_{\mu}(b_{i}). 
%$$
%
First of all, since
the perimeter is additive over sets at strictly positive distance, we observe that $\PP_{([\ve,D-\ve], |\cdot|, \mu)} (E) = \PP_{([\ve,D-\ve], |\cdot|, \mu)} (\cup_{i\in I_{\ve}} [a_{i},b_{i}])$, for every  finite subset  $\ve > 0$. It follows that 
\begin{equation}\label{eq:PPgeqsumi}
\PP_{([0,D], |\cdot|, \mu)} (E) \geq \sum_{i\in \N} \PP_{([0,D], |\cdot|, \mu)} ([a_{i},b_{i}]).
\end{equation}
Now we show that when $[a_{i},b_{i}] \subset (0,D)$ (the general case follows similarly):
\begin{equation}\label{eq:Paibi}
\PP_{([0,D], |\cdot|, \mu)} ([a_{i},b_{i}]) = h_{\mu}(a_{i}) + h_{\mu}(b_{i}). 
\end{equation}
First notice that by taking  $u_{n}$ constantly equal to $1$ on $[a_{i},b_{i}]$, equal to  $0$ on $[0,a_{i} - 1/n]$ and $[b_{i}+1/n,D]$, and 
affine elsewhere, we get 
$$
\int_{\supp(\mu)} |u_{n}(x)'|h_{\mu}(x) dx=\frac{1}{n}\left( \int_{[a_{i}-1/n, a_{i}]} h_{\mu}(x)dx + \int_{[b_{i}, b_{i}+1/n]} h_{\mu}(x)dx \right),
$$
yielding, by the continuity of $h_{\mu}$, $\PP_{([0,D], |\cdot|, \mu)} ([a_{i},b_{i}]) \leq h_{\mu}(a_{i}) + h_{\mu}(b_{i})$.

Consider now any sequence of Lipschitz functions $u_{n}$ converging to $\chi_{[a_{i}, b_{i}]}$ in $L^{1}(\mu)$ and observe that
without loss of generality we can assume that $0 \leq u_{n}\leq 1$ (otherwise we can truncate the function finding a sequence with lower energy).
For the same reason, possibly taking a subsequence $u_{n} \to \chi_{[a_{i},b_{i}]}$ pointwise a.e., we can assume that 
$$
u_{n}^{-1}((0,1/n)) = [0,x_{n}^{-}) \cup (x_{n}^{+} , D]
$$
and 
$$
u_{n}^{-1}((1-1/n,1] )  = (y_{n}^{-}, y_{n}^{+}), 
$$
for some $x_{n}^{\pm}, y_{n}^{\pm}$ with $x_{n}^{-} < y_{n}^{-} < y_{n}^{+} < x_{n}^{+}$, and $x_{n}^{-}, y_{n}^{-} \to a_{i}$ and $x_{n}^{+}, y_{n}^{+} \to b_{i}$. 
Finally we may also assume $|u_{n}'| = u_{n}'$ in $(x_{n}^{-}, y_{n}^{-})$ and $|u_{n}'| = -u_{n}'$ in $(x_{n}^{+}, y_{n}^{+})$.
Now definining 
$$
\hat u_{n} (x) : = \begin{cases}
1/n, & [0,x_{n}^{-})\cup (x_{n}^{+},D]; \\
1-1/n, & (y_{n}^{-},y_{n}^{+}); \\ 
u_{n}, & \textrm{elsewhere},
\end{cases}
$$
we obtain an approximating function with $|\hat u_{n}'| \leq | u_{n}'|$, $\L^{1}$-a.e. over $[0,D]$. Integrating by parts, it follows straightforwardly that 
$$
\lim_{n\to \infty} \int_{\supp(\mu)} |\hat u_{n}'(x)|h(x)dx = h_{\mu}(a_{i}) + h_{\mu}(b_{i}).
$$
Therefore \eqref{eq:Paibi} is proved and it follows that $\PP_{([0,D], |\cdot|, \mu)} (E) \geq \sum_{i\in \N} h_{\mu}(a_{i}) + h_{\mu}(b_{i})$.

To show the converse inequality, consider $u_{n}$ to be the affine approximations of $\sum_{i \leq n}\chi_{[a_i,b_i]}$, constructed as follows: 
assume that $a_{i} > b_{i+1} > a_{i+1}$ and $a_{i}, b_{i} \to 0$ as $i \to \infty$; then for each $n$ consider $m(n) \in \N$ such that 
$$
 a_{i} - b_{i+1} > \frac{2}{m(n)}, \qquad \textrm{for each } i \leq n.
$$
and such that 
\begin{equation}\label{E:m(n)}
\lim_{n\to \infty} \frac{n L_{n}}{m(n)} = 0,
\end{equation}
where $L_{n}$ is the Lipschitz constant of $h_{\mu}$ restricted on $[b_{n+1}, b_{1} + (D-b_{1})/2]$.
Then define $u_{n}$ to be equal 1 on $\cup_{i \leq n} [a_{i},b_{i}]$, to be $0$ on 
$$
[0, a_{n} - 1/m(n)] \bigcup_{i< n} [b_{i+1}+ 1/m(n), a_{i}-1/m(n)]     \cup  [b_{1} +1/m(n),D],
$$
and $u_{n}$ affine elsewhere, so that $u_{n}$ will be Lipschitz.
It follows that 
\begin{align*}
\int_{\supp(\mu)} |u_{n}'(x)| h_{\mu}(x) dx 
&~ = \sum_{i\leq n} m(n) \left( \int_{[a_{i}-1/m(n), a_{i}]} h_{\mu}(x)dx + \int_{[b_{i}, b_{i}+1/m(n)]} h_{\mu}(x)dx \right) \\
&~ \leq \sum_{i\leq n}  h_{\mu}(a_{i}) + \frac{1}{m(n)}L_{n} + h_{\mu}(b_{i})  \\
&~ \leq \frac{nL_{n}}{m(n)} +  \sum_{i\in\N }  h_{\mu}(a_{i}) + h_{\mu}(b_{i}).
\end{align*}
It follows from $\eqref{E:m(n)}$  that 
$$
\PP_{([0,D], |\cdot|, \mu)} (E)  \leq \liminf_{n\to \infty}\int_{\supp(\mu)} |u_{n}'(x)| h_{\mu}(x)\, dx \leq  \sum_{i\in\N }  h_{\mu}(a_{i}) + h_{\mu}(b_{i}).
$$
The claim follows.
%
%
%
%%
%$$
%\int_{\supp\mu} |u_{n}(x)'|h(x)dx  = \sum_{i\in \N} \int_{(\alpha_{i},\beta_{i})} (-1)^{i}u_{n}'(x) h(x) dx,
%$$
%%
%with $\alpha_{i+1} = \beta_{i}$. Now from integration by parts
%%
%$$
%\int_{(\alpha_{i},\beta_{i})} (-1)^{i}u_{n}'(x) h(x) dx = (-1)^{i} [h(\beta_{i}) u_{n}(\beta_{i})  -  h(\alpha_{i}) u_{n}(\alpha_{i})  ] -
%\int_{(\alpha_{i},\beta_{i})} (-1)^{i}u_{n}(x) h'(x) dx
%$$
%%
%
%\medskip
%The fact that $h_{\mu}>0$ on $(0,D)$ readily implies that if 
%$\PP_{([0,D], |\cdot|, \mu)} (E)<\infty$   then $\bar{E}$ is a countable union of intervals, $E=\cup_{i\in\N} [a_i, b_i]$ and 
%%
%$$
%\PP_{([0,D], |\cdot|, \mu)} (E)=\sum_{i\in\N} \big( h_\mu(a_i) + h_{\mu} (b_i) \big).
%$$
%
\end{proof}

We then obtain the following

\begin{corollary}\label{C:IKND}
Let $\mu= h_{\mu} \mathcal{L}^{1} \in \mathcal{F}^{s}_{K,N,D}$, then for any $v \in [0,1]$
\begin{equation}\label{E:perimeter-I}
\inf \{ \PP_{([0,D], |\cdot|, \mu)} (E)\, :\, E\subset [0,D], \;  \mu(E)=v\}=  \inf \{ \mu^+(E)\, :\, E\subset [0,D], \;  \mu(E)=v\} \geq \mathcal{I}_{K,N,D}(v).
\end{equation}
\end{corollary}

\begin{proof}
%First observe that 
%%
%\begin{align*}
%\inf \{ \PP_{([0,D], |\cdot|, \mu)} (E)\, &~:\, E\subset [0,D], \;  \mu(E)=v\}  \\
%&~= \inf \{ \PP_{([0,D], |\cdot|, \mu)} (E)\, :\, E\subset [0,D], \;  \mu(E)=v, \PP_{([0,D], |\cdot|, \mu)} (E) < \infty\};
%\end{align*}
%%
%moreover,  
By Proposition \ref{prop:P1D}, for each $E \subset [0,D]$ of finite perimeter there exists $F = \cup_{i \in \N} [a_{i},b_{i}]$ such that $\mu(E \Delta F) = 0$ 
and 
$$
\PP_{([0,D], |\cdot|, \mu)}(E)  = \PP_{([0,D], |\cdot|, \mu)}(F) = \sum_{i\in \N} h_{\mu}(a_{i}) + h_{\mu}(b_{i}).
$$
We consider  for each $n \in \N$ the family $\mathcal{E}_{n}$ of sets $E$ of finite perimeter admitting a representative 
$F$ made of at most $n$ disjoint closed intervals. Since
\begin{align}
\inf \{ \PP_{([0,D], |\cdot|, \mu)} (E)\, &~:\, E\subset [0,D], \;  \mu(E)=v\}  \nonumber \\
&~= \inf_{n\in \N} \inf_{E \in \mathcal{E}_{n}} \{ \PP_{([0,D], |\cdot|, \mu)} (E)\, :\, E\subset [0,D], \;  \mu(E)=v \}, \label{PPEn}
\end{align}
and  on finite unions of closed intervals the Minkowski content and the perimeter coincide, it follows that 
$$
\PP_{([0,D], |\cdot|, \mu)} (E)  \geq    \inf \{ \mu^+(B)\, :\, B\subset [0,D], \;  \mu(B)=v\}.
$$
To obtain the reverse inequality, just observe that for 
each set $E$ of finite perimeter, by removing the boundary points to the intervals $[a_i,b_i]$, we can also take $F$ open. 
Then from Proposition \ref{prop:P<m+} it follows that 
$$
\PP_{([0,D], |\cdot|, \mu)} (E) = \PP_{([0,D], |\cdot|, \mu)} (F) \leq \mm^{+}(F)  \leq \mm^{+}(E). 
$$
Taking the inf, equality in \eqref{E:perimeter-I} follows. To prove the inequality in \eqref{E:perimeter-I}, just recall that it is one of 
the main results of \cite{Mil}.
\end{proof}

\medskip
\section{Proof of the main results} \label{S:main}

%\begin{theorem}\label{T:iso}
%Let $(X,\sfd,\mm)$ be a metric measure space with $\mm(X)=1$, 
%verifying  the essentially non-branching property and $\CD_{loc}(K,N)$ for some $K\in \R,N \in [1,\infty)$.
%Let $D$ be the diameter of $X$, possibly assuming the value $\infty$.
%\medskip
%
%Then for every Borel subset $E\subset X$, calling $\mm(E)=v\in [0,1]$ it holds
%%
%$$
%\PP(E) \ \geq \ \cI_{K,N,D}(v), 
%$$
%%
%where $\cI_{K,N,D}$ is the model isoperimetric profile  defined in \eqref{eq:defIKND}.
%\end{theorem} 

\begin{proof}[Proof of Theorem \ref{T:iso}] 

First of all we can assume $D<\infty$ and therefore $\mm \in \mathcal{P}_{2}(X)$: indeed from the Bonnet-Myers Theorem if $K>0$ then $D<\infty$, 
and if $K\leq 0$ and $D=\infty$ then the model isoperimetric profile trivializes, i.e. $\cI_{K,N,\infty}\equiv 0$ for $K\leq 0$.  
Also, for $v=0,1$ one has $\cI_{K,N,D}(0)=\cI_{K,N,D}(1)=0$, so again there is nothing to prove.  
Therefore, without loss of generality we can assume $v=\mm(E)\in (0,1)$.

Let $\{u_n\}_{n \in \N} \subset \Lip(X)$ be such that
\begin{equation}\label{eq:defun}
\PP(E)=\lim_{n\to \infty} \int_X |\nabla u_n| \, \mm, \quad u_n \to \chi_E \text{ in } L^1(X,\mm).
\end{equation}

Consider the $\mm$-measurable function $f(x) : = \chi_{E}(x)  - v$ and notice that  $\int_{X} f \, \mm = 0$. 
Thus $f$ verifies the  hypothesis of Theorem \ref{T:localize} and noticing that $f$ is never null, 
we can decompose $X = Y \cup \mathcal{T}$ with 
$$
\mm(Y)=0, \qquad   \mm\llcorner_{\mathcal{T}} = \int_{Q} \mm_{q}\, \qq(dq), 
$$
with $\mm_{q} = g(q,\cdot)_\sharp \left( h_{q} \cdot \mathcal{L}^{1}\right)$; 
moreover,  for $\qq$-a.e. $q \in Q$,  the density $h_{q}$ verifies \eqref{E:curvdensmm}  and 
$$
\int_{X} f(z) \, \mm_{q}(dz) =  \int_{\dom(g(q,\cdot))} f(g(q,t)) \cdot h_{q}(t) \, \mathcal{L}^{1}(dt) = 0.
$$
Therefore 
\begin{equation}\label{eq:volhq}
v=\mm_{q} ( E \cap \{ g(q,t) : t\in \R \} ) = (h_{q}\mathcal{L}^1) (g(q,\cdot)^{-1}(E)), \quad \text{ for $\qq$-a.e. $q \in Q$}. 
\end{equation}
\noindent
Observing that the map $\dom(g(q,\cdot)\ni t \mapsto u_n \circ g(q,t) \in \R$ is Lipschitz and therefore differentiable $\mathcal{L}^1$-a.e., we get that   $|\nabla u_n| (g(q,t))\geq  \left| \frac{d}{dt} (u_n \circ g(q,t)) \right|$ for $\mathcal{L}^1$-a.e. $t\in \dom(g(q,\cdot)$.
This implies
\begin{align*}
\int_X |\nabla u_n|(x) \, \mm(dx)   	&~  =    \int_{\mathcal{T}}  |\nabla u_n|(x) \, \mm(dx)  \,\mm(dx)   \crcr
                                                          &~  =    \int_{Q} \left( \int_{X}    |\nabla u_n|(x) \, \mm_{q} (dx) \right)\, \qq(dq) \crcr
						        &~  \geq     \int_{Q} \left( \int_{\dom(g(q,\cdot))}   \left| \frac{d}{dt} (u_n \circ g(q,t)) \right|  \,h_{q}(t) \, \mathcal{L}^{1}(dt) \right)\, \qq(dq) .
\end{align*}
Now we note that for $\qq$-a.e. $q \in Q$, it holds $u_n \circ g(q,\cdot) \to \chi_{g(q,\cdot)^{-1}(E)}$ in $L^1_{loc} (\dom(g(q,\cdot)),  h_{q}\, \mathcal{L}^{1})$, therefore passing to the limit as $n \to \infty$ in the last inequality, using Fatou's Lemma and  the definition of perimeter we get 
\begin{align*}
\PP(E) &~ =\lim_{n\to \infty}  \int_X |\nabla u_n|(x) \, \mm(dx)    \crcr
           &~  \geq     \int_{Q} \left( \int_{\dom(g(q,\cdot))}   \liminf_{n\to \infty}  \left| \frac{d}{dt} (u_n \circ g(q,t)) \right|  \,h_{q}(t) \, \mathcal{L}^{1}(dt) \right)\, \qq(dq) \crcr
           &~  \geq     \int_{Q} \left( \int_{\dom(g(q,\cdot))}   \cP_{(\dom(g(q,\cdot), |\cdot|, h_{q}\, \mathcal{L}^{1})  }  (g(q,\cdot)^{-1}(E) )  \right)\, \qq(dq) .
\end{align*}
But by construction $(h_{q}\, \mathcal{L}^{1})(g(q,\cdot)^{-1}(E) )=v$ for $\qq$-a.e. $q \in Q$,  therefore thanks to Corollary \ref{C:IKND} 
we infer that
$\PP_{(\dom(g(q,\cdot), |\cdot|, h_{q}\, \mathcal{L}^{1})  }  (g(q,\cdot)^{-1}(E) ) \geq \cI_{K,N,D}(v)$. 
We conclude that
$$
\PP(E) \geq   \int_{Q}   \cI_{K,N,D}(v)  \qq(dq)=  \cI_{K,N,D}(v), 
$$
since $\qq(Q)=1$.
%						        
%        &~ =    \int_{Q} \left(  \frac{(h_{q}\mathcal{L}^1)(g(q,\cdot)^{-1}(A^\ve))-  (h_{q}\mathcal{L}^1)(g(q,\cdot)^{-1}(A))}{\ve}  \right)\, \qq(dq) \crcr  	
%	&~ \geq    \int_{Q} \left(  \frac{(h_{q}\mathcal{L}^1)((g(q,\cdot)^{-1}(A))^\ve)-  (h_{q}\mathcal{L}^1)(g(q,\cdot)^{-1}(A))}{\ve}  \right)\, \qq(dq), \crcr
%
% 
\end{proof}

\subsection{The Cheeger constant}
Recall that the Cheeger constant $h_{(X,\sfd,\mm)}$ is defined by
$$
h_{(X,\sfd,\mm)}:=\inf\left\{\frac{\PP(E)}{\mm(E)} \; : \; E\subset X \text{ is Borel  and } \mm(E)\in (0, 1/2]   \right\}.
$$
In analogy with the model isoperimetric profile, we can also define a model Cheeger constant as follows. The model Cheeger constant for spaces having Ricci curvature bounded below by $K\in \R$, dimension bounded above by $N\geq 1$ and diameter at most $D\in (0,\infty]$ is defined by
\begin{equation}\label{eq:defhKND}
h_{K,N,D}:=\inf_{H\in \R,a\in [0,D]} h_{\left([-a,D-a], J_{H,K,N}\right)},
\end{equation}
where $J_{H,K,N}$ was defined in \cite{Mil} (see also \cite{CM1}), see also thereafter for a more explicit form.

In \cite[Section 5]{CM2}, the authors studied the variant of the Cheeger constant when the perimeter is replaced by the outer Minkowski content and used the results of \cite{CM1} to infer sharp comparison and almost rigidity. Since after this short note  we have at disposal the same results expressed in terms of the perimeter (which we remark are a priori stronger), we can repeat verbatim the proofs of \cite[Theorem 5.3, Corollary 5.3]{CM2} just replacing the outer Minkowski content by the perimeter and obtain the following results.

\begin{theorem}\label{T:cheeger}
Let $(X,\sfd,\mm)$ be an  essentially non-branching $\CD_{loc}(K,N)$-space  for some $K\in \R, N \in [1,\infty)$,  with  $\mm(X)=1$ 
and  having diameter  $D\in (0,+\infty]$. Then
\begin{equation}\label{eq:hcomp}
h_{(X,\sfd,\mm)}\geq h_{K,N,D}.
\end{equation}  
Moreover, for $K>0$ the following holds: for every $N>1$ and $\ve>0$ there exists $\bar{\delta}=\bar{\delta}(K,N,\ve)$ such that,  
for every $\delta\in [0,\bar{\delta}]$, if $(X,\sfd,\mm)$ is an essentially non-branching $\CD_{loc}(K-\delta,N+\delta)$-space such that 
\begin{equation}\label{eq:hdelta}
h_{(X,\sfd,\mm)}\leq h_{K,N,\pi\sqrt{(N-1)/K}}+\delta\quad (=h(S^N(\sqrt{(N-1)/K)})+\delta \text{  if $N\in {\mathbb N}$}), 
\end{equation}
then  ${\diam}(X)\geq \pi\sqrt{(N-1)/K}-\ve$.
\end{theorem}

Before stating the last  result let us observe that if $(X,\sfd,\mm)$ is an $\RCD^*(K,N)$ space for some $K>0$ then, 
called $\sfd':=\sqrt{\frac{K}{N-1}} \; \sfd$, we have that $(X,\sfd',\mm)$ is  $\RCD^*(N-1,N)$; 
in other words, if the Ricci lower bound is $K>0$ then up to scaling we can assume it is actually equal to $N-1$.

\begin{corollary}\label{C:almostcheeger}
For every $N\in [2, \infty) $,  $\ve>0$ there exists $\bar{\delta}=\bar{\delta}(N,\ve)>0$ such that the following hold. 
For every  $\delta \in [0, \bar{\delta}]$, if  $(X,\sfd,\mm)$ is an $\RCD^*(N-1-\delta,N+\delta)$-space with $\mm(X)=1$, satisfying 
$$
h_{(X,\sfd,\mm)}\leq h_{N-1,N,\pi}+\delta \quad (=h(S^N)+\delta \text{  if $N\in {\mathbb N}$}), 
$$
then  there exists an $\RCD^*(N-2,N-1)$ space $(Y, \sfd_Y, \mm_Y)$ with $\mm_Y(Y)=1$ such that 
$$
\sfd_{mGH}(X, [0,\pi] \times_{\sin}^{N-1} Y) \leq \ve. 
$$
In particular, if $(X,\sfd,\mm)$ is an $\RCD^*(N-1,N)$-space satisfying $h_{(X,\sfd,\mm)}= h_{N-1,N,\pi}(=h(S^N)$ if $N\in \N)$, 
then it is isomorphic to a spherical suspension; i.e. there exists an $\RCD^*(N-2,N-1)$ space $(Y, \sfd_Y, \mm_Y)$ 
with $\mm_Y(Y)=1$ such that $(X,\sfd,\mm)$ is isomorphic to $[0,\pi] \times_{\sin}^{N-1} Y$.
\end{corollary}

\end{document}